\documentclass{article}
\usepackage{amsmath, amsthm, epsfig,amssymb, multicol, tikz}

\usepackage{xcolor}
\usepackage{lipsum}
\usepackage{amsfonts,amsmath,amssymb,amsthm}
\usepackage[shortlabels]{enumitem}
\usepackage{graphicx}

\usepackage{url}
\usepackage{hyperref}

\newtheorem{lemma}{Lemma}
\newtheorem{theorem}{Theorem}
\newtheorem{corollary}{Corollary}
\newtheorem{definition}{Definition}

\title{Spectral Radius of $\{0, 1\}$-Tensor with Prescribed Number of Ones}
\author{
Shuliang Bai \thanks{University of South Carolina, Columbia, SC 29208,
({\tt sbai@math.sc.edu}).} \and 
Linyuan Lu
\thanks{University of South Carolina, Columbia, SC 29208,
({\tt lu@math.sc.edu}). This author was supported in part by NSF
grant DMS 1600811 and ONR grant N00014-17-1-2842.} }
\begin{document}
\maketitle

\begin{abstract}
  For any $r$-order $\{0, 1\}$-tensor $A$ with $e$ ones,
we prove that the spectral radius of $A$ is at most $e^{\frac{r-1}{r}}$
with the equality holds if and only if $e={k^r}$ for some integer $k$ and
all ones  forms a principal sub-tensor ${\bf 1}_{k\times  \cdots \times k}$.
We also prove a stability result for general tensor $A$ with $e$ ones where
$e=k^r+l$ with relatively small $l$. Using the stability result,
we completely characterized 
the tensors achieving the maximum spectral radius
among all
$r$-order $\{0, 1\}$-tensor $A$ with $k^r+l$ ones, for $-r-1\leq l \leq r$,
and  $k$ sufficiently large. 
\end{abstract}

\section{Introduction}
For a real nonnegative square matrix $A$ the spectral radius $\rho(A)$
is the largest eigenvalue of $A$ in modulus,
which is real as guaranteed by
the Perron-Frobenius theorem. The problem of finding the maximal spectral radius for
all $\{0,1\}$-matrices with prescribed
number of ones was introduced by Brualdi and Hoffman \cite{Brualdi} in 1985.
Let $g(e)$ be the maximal spectral radius of $A$ among all $\{0,1\}$-matrices 
$A$ with $e$ ones. They proved that for each positive integer $k$, 
$g(k^2)=g(k^2+1)=k$.
When $e=k^2$, the equality holds if $A$ is essentially a
$k\times k$ all-$1$-matrix (inserted by possibly extra rows/columns of $0$'s).  
When $e=k^2+1$ and $k\geq 3$, the equality is attained for only
when a useless additional $1$ is put at any place else to a $k\times k$ all-$1$-matrix.
(But for $k=1$, or $2$, there is another
$A$ with $\rho(A)=k$.)  Friedland \cite{Friedland} solved another cases
when $e=k^2-1$, $e=k^2-4$,
or $e=k^2+l$ for a fixed $l$ and $k$ sufficiently large. In all cases, the matrices with maximum spectral radius are characterized.

In this paper, we consider a similar problem for 
$\{0,1\}$-tensor (of order $r>2$) with
a fixed number of $1$'s. We ask which tensor attains the maximum spectral radius.

An $n$-dimension $r$-order tensor $A$ in real field $\mathbb{R}$ is 
a multi-dimensional array consisting of $n^r$ entries: 
$$a_{i_1\cdots i_r}\in \mathbb{R},~ \mbox{ where indexes }
i_1,i_2,\ldots, i_r \mbox{ ranges from 1 to n}.$$
 $A$ is called {\em nonnegative} if every element $a_{i_1\cdots i_r}\geq 0$;
 it is called {\em symmetric} if its entries are invariant under any permutation of their indices, i.e. $a_{i_1\cdots i_r}=a_{i_{\sigma(1)}\cdots i_{\sigma(r)}}$ for all $\sigma\in \mathfrak{S}_r$, where $\mathfrak{S}_r$ is a symmetric group on $[r]$.   For every $i \in [n]$, the {\em $i$th slice $A_i$} is an sub-tensor of $A$
 consisting of all elements  $a_{ii_2\cdots i_r}$ with the first index
 being fixed to $i$.

 For a tensor $A$ of order $r\geq 2$ and dimension $n \geq 2$,
a pair $(\lambda, \mathbf{x})\in \mathbb{C}\times (\mathbb{C}^n\setminus\{\mathbf{0}\})$ is called an {\em eigenvalue} and an {\em eigenvector} of $A$, if they satisfy $$A\mathbf{x}^{r-1}=\lambda \mathbf{x}^{[r-1]}~~~~\text{where}~~~~\mathbf{x}^{[r-1]}=(x_1^{r-1},\ldots, x_n^{r-1})^T. $$ 
 That is, for all $i=1,2,\ldots, n$,
 \begin{align}\label{eigenvalue}
 \sum\limits_{i_2,...,i_r=1}^{n} a_{ ii_2 \cdots i_r} x_{i_2} \cdots x_{i_r}=\lambda x_{i}^{r-1}.
 \end{align}
 
 The spectral radius $\rho(A)$ is defined to be the largest modulus of 
 eigenvalues of $A$. $$\rho(A)=\max \{|\lambda| :\text{ $\lambda$ is an eigenvalue of $A$} \}.$$

If $\mathbf{x}$ is a real eigenvector of $A$, clearly the corresponding
eigenvalue $\lambda$ is also real. In this case, $\mathbf{x}$ is called an $H$-eigenvector and $\lambda$  an $H$-eigenvalue.
Furthermore, if $\mathbf{x}\in\mathbb{R}^n_+$,
where $\mathbb{R}^n_+ = \{x \in \mathbb{R}^n : x \geq 0\}$, 
then $\lambda$ is an $H^+$-eigenvalue of $A$. 
If $\mathbf{x}\in \mathbb{R}^n_{++}$, 
where $\mathbb{R}^n_{+ +}= \{x \in \mathbb{R}^n : x > 0\}$,
then $\lambda$ is said to be
an $H^{++}$-eigenvalue of $A$.

The classical Perron-Frobenius theorem for matrix has been generalized to nonnegative tensors:
\begin{theorem}\label{propertiesfornongegative}
\textbf{(Perron-Frobenius theorem for nonnegative tensors)}
\begin{enumerate} 
\item (Yang and Yang 2010 \cite{YangYang}) If $A$ is nonnegative tensor of order $r$ and dimension $n$, 
then the spectral radius $\rho(A)$ is an $H^+$-eigenvalue of $A$. 
\item (Frieland, Gaubert and Han 2011 \cite{FriedlandGaubetHan}) If furthermore $A$ is weakly irreducible, 
then  $\rho(A)$ is the unique  $H^{++}$-eigenvalue of $A$, 
with the unique eigenvector $\mathbf{x}\in \mathbb{R}^n_{++}$, up to a positive scaling coefficient.
\item (Chang, Pearson and Zhang 2008 \cite{ChangPearsonZhang}) If moreover $A$ is irreducible, then  $\rho(A)$ is the unique  $H^{+}$-eigenvalue of $A$, 
with the unique eigenvector $\mathbf{x}\in \mathbb{R}^n_{+}$, up to a positive scaling coefficient.
\end{enumerate}
\end{theorem}

In this paper, we will focus on $\{0, 1\}$-tensors, in which every entry
$a_{i_1\cdots i_r}$ is either $0$ or $1$.  An $n$-dimension $r$-order $\{0, 1\}$-tensor $A$ can be viewed
as a general linear-ordered hypergraph $H=(V,E)$, where $V=[n]$ and $E$ consists
of all $r$-tuples $(i_1,i_2,\ldots, i_r)$ such that $a_{i_1i_2\cdots i_r}=1$.
A $\{0, 1\}$-tensor is always nonnegative,  
thus the spectral radius $\rho(A)$ is an $H^+$-eigenvalue, and 
the associated eigenvector $\mathbf{x}\in \mathbb{R}^n_{+}$.

Consider the set of all $\{0, 1\}$-tensors with a fixed number
of 1's. For fixed integer $r\geq 3$ and $e\geq 1$, let
$${\cal T}^r_{n,e}=\{\mbox{ all  $\{0, 1\}$-tensors of order $r$ and dimension $n$
  with exactly $e$ 1's}\},$$
and $${\cal T}^r_{e}=\cup_n{\cal T}^r_{n,e}.$$

Now we consider the objective function
$$g_r(e)=\max_{A\in {\cal T}^r_e}\rho(A).$$
For a fixed $r$ and $e$, we say $A\in {\cal T}^r_e$ is a {\em maximum tensor} if
$\rho(A)=g_r(e)$.
We are interested in the following questions:
\begin{enumerate}
\item {\it What are the values of $g_r(e)$? Can we prove a tight upper bound?}
\item {\it What does $A$ look like if $\rho(A)$ is very closed to $g_r(e)$?}
\end{enumerate}

There are several operations on ${\cal T}^r_e$ that keep both the spectral radius and the number of $1$'s.
\begin{description}
\item[Permutation on vertices (\cite{LLLL}\label{LLLL}):] For any permutation $\varphi\in \mathfrak{S}_n$ and any tensor
  $A=(a_{i_1i_2\cdots i_r}) \in {\cal T}^r_{n}(e)$,
    define a new tensor as follows: $$\varphi(A)=(a_{\varphi(i_1)\varphi(i_2)\cdots \varphi(i_r)}).$$

\item[Transpose on indexes greater than $1$:] For any permutation $\tau$ on the
      index set $\{2,3, \ldots, r\}$, define a new tensor $A_\tau$  as follows:
      $$A_\tau=(a_{i_1i_{\tau(2)}\cdots i_{\tau(2)}}).$$
    
\item[Deleting/inserting isolated vertices:]
        An index/vertex $v$ is called {\em isolated} if $a_{i_1i_2\cdots i_r}=0$ as long as $v$
        appears in the index $\{i_1,i_2,\ldots, i_r\}$. Deleting/Inserting an isolated vertex
        keeps the spectral radius.
 \end{description}

 We say two tensors in ${\cal T}^r_e$ are {\em equivalent} if one can be obtained from the other one by
 a sequence of the above operations. Denote $J^r_k$ as the $k$-dimension $r$-order all-1-tensor 
 ${\bf 1}_{k\times \ldots \times k}$, 
 it plays a special role in the maximum tensors.

In this paper, we prove the following theorem:
\begin{theorem}\label{t1}
  For any $r$-order $\{0, 1\}$-tensor $A$ with $e$ ones, the spectral radius $\rho(A)$
  satisfies
  $$\rho(A)\leq e^{\frac{r-1}{r}},$$
  with the equality holds if and only if $e=k^r$ for some positive integer $k$ and $A$ is equivalent to 
  $J^r_k$.
\end{theorem}

We also characterize the structure of maximum tensors for $e=k^r+l$ with sufficiently large $k$ and
 $l\in \{-r-1, -r,\ldots,-1,0,1,2,\ldots, r\}$.

\begin{theorem}\label{completlycharctersmalll}
Let $r,  k$ be positive integers with $r\geq 3$ and $k$ sufficiently large.
\begin{enumerate}
\item For $e=k^r+1$, the maximum tensors in  ${\cal{T}}_e^r$ are exactly the tensors which can
  be obtained from $J^r_k$ by inserting an $1$ to an arbitrary $0$-position. All these maximum tensors
  have spectral radius $k^{r-1}$.
  
\item For $2\leq l\leq r$,  $e=k^r+l$, the maximum tensors in  ${\cal{T}}_e^r$ is uniquely
  equivalent to the tensor obtained from $J_k^r$ by inserting $l$ ones at first $l$ positions of the  list:
  $$\{a_{(k+1)11\cdots 1}, a_{1(k+1)1\cdots 1}, a_{11(k+1)\cdots 1}, \ldots, a_{11\cdots (k+1)}\}.$$
  
\item For $1\leq l\leq r+1$,  $e=k^r-l$, the maximum tensors in  ${\cal{T}}_e^r$ is uniquely equivalent to 
the tensor obtained from $J_k^r$ by 
placing $l$ zeros at the first  $l$ positions from the list:
$$\{a_{kk\cdots k}, a_{k(k-1)k\cdots k}, a_{kk(k-1)\cdots k}, \ldots, a_{kkk\cdots (k-1)}, a_{(k-1)k\cdots k} \}.$$
\end{enumerate}

\end{theorem}

A special symmetric tensor, the adjacency tensor $A(H)$ of an $r$-uniform hypergraph $H$ on $n$ vertices is defined as $A=(a_{i_1 \cdots i_r})$ where
$a_{i_1 \cdots i_r}=\frac{1}{(r-1)!}$ if $\{i_1, \ldots, i_r\}\in E(H)$, and equals $0$ otherwise.
In our previous paper\cite{LuBai}, 
we gave a bound on spectral radius of $r$-uniform hypergraph with $e$ edges using an $\alpha$-normal labeling method\cite{LuMan}, which is $\rho(H)\leq f_r(e)$, where 
$f_r(x)$ is a function such that 
$f_r\left({n\choose r}\right)={n-1\choose r-1}$.   
The equality holds if and only if $e={k\choose r}$, for integers $k, r$ and $k\geq r$.
Although the results (of two papers) are comparable, the methods are quite different.
\vspace{5mm}

The paper is organized as follows: In section 2, we prove some important lemmas on nonnegative tensors. In section 3, we prove Theorem \ref{t1} and also give lower bounds of the spectral radius. In section 4, we show the structure of the maximum $\{0, 1\}$-tensor when $e=k^r+l$ with relatively small $l$. In section 5, we determine
the maximum tensors for $-r-1\leq l\leq r$ to finish the proof of Theorem \ref{completlycharctersmalll}.

\section{Lemmas on nonnegative tensors}
In this section, we will prove important properties for nonnegative tensors. We start with some definitions and known facts. 
\begin{definition}
\cite{Lim} An n-dimension r-order tensor $A = (a_{i_1i_2\cdots i_r})$ is called {\em reducible}
 if there exists a nonempty proper subset $I  \subset \{1,\ldots, n\}$ such that
 $a_{i_1i_2\cdots i_r}=0$ for all $i_1 \in I$ and $i_2, \ldots , i_r \notin I$.
 A tensor A is said to be irreducible if it is not reducible.
\end{definition}

\begin{definition}
\cite{HHQ}  A nonnegative matrix $G(A)$ is called the representation associated to the nonnegative tensor $A$, if the
$(i, j)$-th element of $G(A)$ is defined to be the summation of $a_{ii_2\cdots i_r}$ with indices 
$j\in \{i_2,\ldots, i_r\}$.
A nonnegative tensor $A = (a_{i_1i_2\cdots i_r})$ is said to be {\em weakly reducible}
if $G(A)$ is a reducible matrix. It is {\em weakly irreducible} if it is not weakly
reducible.
\end{definition}

\begin{theorem}
\cite{FriedlandGaubetHan, PZ} For an n-dimension r-order tensor $A = (a_{i_1i_2\cdots i_r})$, let $G_A = (V (G_A), E(G_A))$
be the digraph of the tensor A with vertex set $V (G_A) = \{1, 2, \ldots, n\}$ and arc
set $E(G_A) = \{(i, j)| a_{ii_2\cdots i_m}\neq 0, j \in {i_2,\ldots, i_m}\}$.  $A$ is weakly irreducible if the corresponding directed graph $G(A)$ is strongly connected. That is for any pair of vertices 
$i$ and $j$, there exist directed paths from $i$ to $j$ and $j$ to $i$.
\end{theorem}

\begin{theorem}\cite{Shao}\label{decompositionoftensor}
Let A be an n-dimension r-order tensor, $r\geq 2$. Then there exists positive integers $k\geq 1$ and $n_1, \ldots, n_k$ with $n_1+\cdots +n_k=n$ such that A is permutational similar to some $(n_1, \ldots, n_k)$-lower triangular block tensor, where all the diagonal blocks $A_1, \ldots, A_k$ are weakly irreducible. And we have:
\begin{align*}
& Det (A) =\prod\limits_{i=1}^{r} (Det A_i)^{(r-1)^{n-n_i}},
\end{align*}
and thus
\begin{align*}
& \phi_A(\lambda)=\prod\limits_{i=1}^{r} (\phi_{A_i}(\lambda))^{(r-1)^{n-n_i}}
\end{align*}
where $\phi_A(\lambda)$ is the characteristic polynomial of the tensor A, that is $\phi_A(\lambda)=Det(\lambda I-A)$.
\end{theorem}

Please refer to \cite{Shao} for more details on the definitions of determinants and the characteristic polynomial of tensor $A$. Since $\lambda$ is an eigenvalue of $A$ if and only if it is a root of the
characteristic polynomial of $A$,
Theorem \ref{decompositionoftensor} says that the spectral radius of tensor $A$ is the spectral radius of lower triangular block tensor $A_i$ for some $i$. This allows us to consider weakly irreducible tensor only.

We first prove the following lemma on general nonnegative tensors. 
\begin{lemma}\label{rho(A)>k}
  Let $A$ be an $n$-dimension $r$-order nonnegative tensor. If there exists a nonzero vector $\mathbf{x}\in \mathbb{R}^n_{+}$ and a scalar $\lambda$ such that $A\mathbf{x}^{r-1}\geq \lambda\mathbf{x}^{[r-1]}$, then we have $$\rho(A)\geq \lambda.$$
Moreover, if $A$ is weakly irreducible then the equality holds if and only if $\mathbf{x}$ is an eigenvector corresponding to $\rho(A)$.
\end{lemma}

Before proving this lemma, we have a simple corollary. Let $A$ and $B$ are two tensors of the same
dimension and the same order. We say $A\geq B$ if $A-B$ is nonnegative. We also write $A>B$
if $A\geq B$ and $A\not=B$.

\begin{corollary}\label{subtensor}
  For any two nonnegative tensors $A$ and $B$, if $A\geq B$, then $\rho(A)\geq \rho(B)$.
  Furthermore, if $B$ is weakly irreducible and $A>B$, then $\rho(A)>\rho(B)$.
\end{corollary}
\begin{proof}
  Let $\mathbf{x}\in \mathbb{R}^n_+$ be the Perron-Fronbenius vector of $B$. Observing
  \begin{equation}
    \label{eq:AB}
 A\mathbf{x}^{r-1}\geq B\mathbf{x}^{r-1}=\rho(B) \mathbf{x}^{[r-1]}.   
  \end{equation}
  Applying Lemma \ref{rho(A)>k}, we have $\rho(A)\geq \rho(B)$.

  If further $B$ is weakly irreducible, then $\mathbf{x}\in \mathbb{R}^n_{++}$.
  Since $A>B$, one of Inequalities \ref{eq:AB} is strict. In particular, $\mathbf{x}$ is
  not an eigenvector of $A$. Thus, we must have $\rho(A)>\rho(B)$.
\end{proof}

\begin{proof}[Proof of Lemma \ref{rho(A)>k}:]
When $\lambda=0$, it is trivial. Without loss of generality, we assume $\lambda>0$.

  First we consider the case when $A$ is weakly irreducible. 
We claim that we can modify $\mathbf{x}$ so that $\mathbf{x}\in \mathbb{R}^n_{++}$. That is, 
if there exists a nonzero vector $\mathbf{x}\in \mathbb{R}^n_{+}$ and a scalar $\lambda$ such that $A\mathbf{x}^{r-1}\geq \lambda\mathbf{x}^{[r-1]}$, then there exists a new vector ${\bf y}\in \mathbb{R}^n_{++}$, such that $A\mathbf{y}^{r-1}\geq \lambda\mathbf{y}^{[r-1]}$. 
 
 If not, let $J=\{j\in [n] \mid x_j=0\}\not=\emptyset$. Let $J_0=J$ and for $i=1,2,\ldots$, define
 $$J_i=\{j\in J_{i-1} \mid a_{jj_2\cdots j_r}>0 \Rightarrow j_2,\ldots, j_r\in J_{i-1}\}. $$
 We have
 $$J=J_0\supseteq J_1 \supseteq J_2\supseteq \ldots.$$
 Assume $J_i$ is stabilized after $s$ steps; i.e., $J_{s}=J_{s+1}$.
 Since $A$ is weakly irreducible, $J_s=\emptyset$.

 Let $\delta$ be the minimum among all positive entries of $A$. 
 Let
 $\epsilon>0$ be a tiny positive number satisfying 
 $\frac{1}{\epsilon}\gg \log(1/\epsilon)\geq \frac{\lambda}{\delta}$.
 For $i=1,2,\ldots, s$, set $a_i=\sum_{j=1}^i(r-1)^{s-j}$, and $\epsilon_i=\frac{\epsilon}{\log^{a_i}(1/\epsilon)}.$ We have
 $$0<\epsilon_s \ll \epsilon_{s-1} \ll \cdots \ll \epsilon_{2} \ll \epsilon_1 \ll \epsilon\ll
 \frac{\delta}{\lambda}.$$

 We define a new variable $\mathbf{y}=(y_1,y_2,\ldots, y_n)^{T}\in \mathbb{R}^n_{++}$ by
 $$ y_j=
 \begin{cases}
   x_j \mbox{ if } j\not\in J;\\
   \epsilon_i \mbox{ if } j\in J_{i-1}\setminus J_i.\\
 \end{cases}
 $$
 We claim that $A\mathbf{y}^{r-1}\geq \lambda \mathbf{y}^{[r-1]}$.
 When $j\not\in J$, we have
 $$(A\mathbf{y}^{r-1})_j\geq (A\mathbf{x}^{r-1})_j\geq \lambda x_j^{r-1}=\lambda y_j^{r-1}.$$
 When $j\in J_{i-1}\setminus J_i$, then there exist an entry $a_{jj_2\ldots j_r}>0$ and at least one index $j_l\not\in J_{i-1}$ ($l\geq 2$).
 Thus, we have
 \begin{align*}
   (A\mathbf{y}^{r-1})_j &\geq a_{jj_2\ldots j_r}y_{j_2}\cdots y_{j_r}\\
                         &\geq \delta \epsilon_{i-1}\epsilon_s^{r-2}\\
                         &= \delta \frac{\epsilon^{r-1}}{\log^{a_{i-1}+(r-2)a_s} (1/\epsilon)}\\
                         &\geq \lambda\frac{\epsilon^{r-1}}{\log^{a_{i-1}+(r-2)a_s+1} (1/\epsilon)}\\
                         &= \lambda  \frac{\epsilon^{r-1}}{\log^{a_{i}(r-1)} (1/\epsilon)}\\
                         &=\lambda \epsilon_i^{r-1}\\
                         &=\lambda y_j^{r-1}.
                        \end{align*}

 Here we applied the equality
 $$a_{i-1}+(r-2)a_s+1= a_{i}(r-1),$$
 which can be verified directly by the definition of $a_i$.

 Hence, without loss of generality, we can assume $\mathbf{x}\in \mathbb{R}^n_{++}$.
  For any $\lambda>0$, we define two sets $S_{\lambda}$ and $S_{\lambda}^+$ as follows:
  \begin{align*}
    S_{\lambda}&=\{\mathbf{x}\in \mathbb{R}^n_{+}\colon A\mathbf{x}^{r-1}\geq\lambda\mathbf{x}^{[r-1]}\},\\
S_{\lambda}^+&=\{\mathbf{x}\in \mathbb{R}^n_{++}\colon A\mathbf{x}^{r-1}\geq\lambda\mathbf{x}^{[r-1]} \mbox{ and at least one inequality is strict}\}.  
  \end{align*}

Let $\Lambda =\{\lambda\colon S_{\lambda}^+\neq \emptyset\}$. 

\textbf{Claim 1:} $\Lambda\subset \mathbb{R}$ is an open set. 

For any $\lambda\in \Lambda$, there exists $\mathbf{x}\in \mathbb{R}^n_{++}$ satisfying the following system:
\begin{equation}\label{systemequations}
  \sum\limits_{i_2,...,i_r=1}^{n} a_{
    i i_2 \cdots i_r} x_{i_2} \cdots x_{i_r}\geq \lambda x_{i}^{r-1}
  \mbox{ for } i=1,2,\ldots, n.
\end{equation}
Let $A^i$ be the $i$-th equation in (\ref{systemequations}) and $I$ be the index such that the equality holds at $A^i$.
That is,
$I=\{i\in [n] | \sum\limits_{i_2,...,i_r=1}^{n} a_{ ii_2 \cdots i_r} x_{i_2} \cdots x_{i_r}=\lambda x_{i}^{r-1}\}$.

Assume $I\not=\emptyset$. 
Since $G_A$ is strongly connected, there exist at least one pair vertices  $i\in I$ and  $u\in [n]\setminus I$  such that $(i, u)\in E(G_A)$, for this to happen, we have $a_{ ii_2 \cdots i_r}\neq 0$ when $u=i_l$ for some $l\geq 2$.
Then $x_u$ appears in
equation $A^i$. Since $A^u$ is a strictly inequality,  we can add appropriate positive tiny value $\epsilon_u$ to $x_u$ so that $A^u$ remains a strictly inequality.
Now the
$i$-th
equation $A^i$ becomes a strictly inequality while other strictly greater inequalities remain strict.
By induction on $|I|$,
after finite steps, we can obtain a new vector $\mathbf{x'}$ to replace $\mathbf{x}$ and  
we will have a new system with all strictly greater inequalities. That is,
for all $i\in [n]$,
$$\sum\limits_{i_2,...,i_r=1}^{n} a_{ ii_2 \cdots i_r} x'_{i_2} \cdots x'_{i_r}>
\lambda {x'}_{i}^{r-1}.$$
Therefore there exists an $\epsilon>0$ such that $A\mathbf{x'}^{r-1}> (\lambda+\epsilon) \mathbf{x'}^{[r-1]}$. Thus $(\lambda-\epsilon, \lambda+\epsilon) \subset \Lambda$.  $\Lambda$ is an open set. 

Since $\rho(A)$ exists, $\Lambda$ is a bounded set.
Let $\lambda_0=\sup(\Lambda)$.

\textbf{Claim 2:}
$\lambda_0$ is an eigenvalue of $A$. In particular, $\lambda_0\leq \rho(A)$.

In the definition of $S_\lambda$, the system of inequalities are homogeneous in $\mathbf{x}$.
Without loss of generality, we can normalize $\mathbf{x}$ so that $\|\mathbf{x}\|_r=1$.
Note that the sphere in the first quadrant
$\{x\in \mathbb{R}^n_+\colon \|\mathbf{x}\|_r=1\}$ is a compact set.
Thus any sequence has a convergent subsequence and the limit point is also in this set.
It implies that there is a $\mathbf{x}\in  \mathbb{R}^n_+$ so that
$$A \mathbf{x}^{r-1}\geq \lambda_0 \mathbf{x}^{[r-1]}.$$
Now we show that $\mathbf{x}>0$. 
Assume not, let $J=\{i\in [n]\colon x_i=0\}$.
By the previous argument, we can find a $y\in \mathbb{R}^n_{++}$ still satisfying
$$A \mathbf{y}^{r-1}\geq \lambda_0 \mathbf{y}^{[r-1]}.$$
Also notice that $A_i$ are strictly inequality for all $i\in J$. 
Thus $\lambda_0\in \Lambda$. Contradiction to the fact that $\Lambda$ is an open set.

Hence $\mathbf{x}\in  \mathbb{R}^n_{++}$ and $A \mathbf{x}^{r-1}= \lambda_0 \mathbf{x}^{[r-1]}$. Thus $\lambda_0$ is an eigenvalue of $A$. Therefore
$$\lambda\leq \lambda_0\leq \rho(A).$$
If the inequality $\lambda=\rho(A)$ holds, then $\bf x$ is
an eigenvector for $\rho(A)$.

Now we consider general $A$.  By Theorem \ref{decompositionoftensor}, $A$ is  permutationally similar to some $(n_1,n_2,\ldots, n_k)$-lower triangular block tensor, where all the diagonal blocks $A_1, \ldots, A_k$ are weakly irreducible. Denote by $I_i$
the $i$-th block of indexes of size $n_i$.
We have
$$A_{1}({\bf x}|_{I_1})^{r-1}=(A {\bf x}^{r-1})|_{I_1}\geq (\lambda {\bf x}^{[r-1]})|_{I_1}=\lambda ({\bf x} |_{I_1})^{[r-1]}.$$
If $\mathbf{x}|_{I_1}\not=0$,
 then the weakly irreducible tensor $A_1$ satisfies the condition of lemma. Thus by previous argument, we are done:
 $$\rho(A)\geq \rho(A_1)\geq \lambda.$$
 If $\mathbf{x}|_{I_1}=0$, we consider $I_2$, and so on. Let $j$ be the first indexes
 so that $\mathbf{x}|_{I_j}\not=0$.
 We have 
 $$A_{j}({\bf x}|_{I_j})^{r-1}=(A {\bf x}^{r-1})|_{I_j}\geq (\lambda {\bf x}^{[r-1]})|_{I_j}=\lambda ({\bf x} |_{I_j})^{[r-1]}.$$
Now the weakly irreducible tensor $A_j$ satisfies the condition of lemma.
 We still have
 $$\rho(A)\geq \rho(A_j)\geq \lambda.$$
\end{proof} 
Lemma \ref{rho(A)>k}
plays an important role in characterizing the largest eigenvalue and 
thus can be applied to determine the maximum tensors in the last section.
In fact, this lemma gives another proof for the existence of the Perron-Frobenius vector
for nonnegative tensor.
Cooper and Dutle \cite{Cooper} proved a similar result on adjacency tensor of connected uniform hypergraph, that is, on a symmetric nonnegative weakly irreducible tensor.

Next, we will generalize a theorem of Schwarz\cite{Sch} on general nonnegative $r$-order tensors with $r\geq 3$.   For  $n, r\geq 3$,  let $\sigma$ be a given set of $n^r$ nonnegative real numbers   
(not necessarily pairwise distinct) and let $\mathcal{F}(\sigma)$ be the set of all $n$-dimension $r$-order tensors $A$ for which $\sigma$ is the set of their elements. 
Denote $f(\sigma)$ as the largest spectral radius among tensors in $\mathcal{F}(\sigma)$. 
Let  $\mathcal{F}^*(\sigma)$ be the subset of $\mathcal{F}(\sigma)$ consisting of
these tensors having the property that in each slice $A_i$ the elements decrease according to the dictionary order;
i.e. $a_{i i_2 \cdots i_r}\geq a_{i j_2 \cdots j_r}$ whenever $(i_2, \ldots, i_r)\leq (j_2, \ldots, j_r)$
under the dictionary order.
Let $f^*(\sigma)$ be the largest spectral radius among tensors in $\mathcal{F}^*(\sigma)$.
We first show that $f(\sigma)$ is attained by some tensor in $\mathcal{F}^*(\sigma)$.
\begin{theorem}\label{maxima}
  $f(\sigma)=f^*(\sigma)$.
 \end{theorem}

\begin{proof}
Let $A$ be a tensor that attains the largest eigenvalue in $ \mathcal{F}(\sigma)$, i.e.  $\rho(A)=f(\sigma)$. 
Let  $\mathbf{x}=(x_1, \ldots, x_n)\in \mathbb{R}^n_{+}$ be the eigenvector associated to $f(\sigma)$. Since permutation on vertices keeps the spectral radius,
without loss of generality, 
we can assume $x_1\geq x_2\geq \ldots\geq x_n\geq 0$.
Now fix the vertex order of vertices. 

Suppose that $A\not\in \mathcal{F}^*(\sigma)$. Then $A$ contains a
pair of entries $a_{i_1 i_2 \cdots i_r}$
and $a_{i_1 j_2 \cdots j_r}$
satisfying
$$a_{i_1 i_2 \cdots i_r}< a_{i_1 j_2 \cdots j_r}
\mbox{ but } (i_2, \ldots, i_r)< (j_2, \ldots, j_r).$$
We call such pair as a {\em disordered} pair. 

By sequentially switching a disordered pair until no disordered pair is found, we create a sequence of tensors $B_0, B_1, B_2,\ldots, B_s\in \mathcal{F}(\sigma)$ satisfying
\begin{enumerate}
\item $B_0=A$, and $B_s\in \mathcal{F}^*(\sigma)$.
\item For each $k$ from $1$ to $s$,  $B_{k}$ is created from $B_{k-1}$ by switching one disordered pair.
\end{enumerate}
We claim that for each $k$, 
$$B_k\mathbf{x}^{r-1}\geq B_{k-1}\mathbf{x}^{r-1}.$$
Suppose that $(b_{i_1 i_2 \cdots i_r}, b_{i_1 j_2 \cdots j_r})$ is the disordered pair of $B_{k-1}$, which
is switched to create $B_k$.

Then for any $i\neq i_1$, the $i$-th row is not affected by switching:
\begin{align}\label{ineqj}
(B_k\mathbf{x}^{r-1})_i=(B_{k-1}\mathbf{x}^{r-1})_i.
\end{align}

Since  $(b_{i_1 i_2 \cdots i_r}, b_{i_1 j_2 \cdots j_r})$ is a disordered pair,
we have $b_{i_1 i_2 \cdots i_r}<b_{i_1 j_2 \cdots j_r}$ and $(i_2, \ldots, i_r)< (j_2, \ldots, j_r)$.
This implies $x_{j_2}\ldots x_{j_r}\leq x_{i_2}\ldots x_{i_r}$ since
$x_1\geq x_2\geq \cdots \geq x_n>0$. Thus,
for the $i_1$-th row, we have
\begin{align*}
&\hspace*{-2mm} (B_k\mathbf{x}^{r-1})_{i_1}-(B_{k-1}\mathbf{x}^{r-1})_{i_1} 
  &=(b_{i_1 i_2 \cdots i_r}-b_{i_1 j_2 \cdots j_r})\left(x_{j_2}\ldots x_{j_r} -x_{i_2}\ldots x_{i_r}
    \right)\\
  &\geq 0.
\end{align*}
The claim is proved. Therefore, we have
$$B_s\mathbf{x}^{r-1}\geq B_{s-1}\mathbf{x}^{r-1}\geq \cdots \geq  B_{0}\mathbf{x}^{r-1}=A\mathbf{x}^{r-1}
=\rho(A)\mathbf{x}^{[r-1]}.$$
Applying Theorem \ref{rho(A)>k}, we get
$$\rho(B_s)\geq \rho(A).$$
Since $A$ has the maximum spectral radius in $\mathcal{F}(\sigma)$, so is $B_s$.
Thus $f^*(\sigma)=f(\sigma)$.
The proof is finished.
\end{proof}

\textbf{Remark:} 
Note that if we restrict all tensors in $\mathcal{F}(\sigma)$ to be symmetric, we can get a stronger condition on
the maximum tensor $A$: $a_{i_1 i_2 \cdots i_r}\geq a_{j_1 j_2 \cdots j_r}$ whenever $(i_1, \ldots, i_r)\leq (j_1, \ldots, j_r)$.  The proof is easy, we only need to use the fact that the spectral radius of symmetric tensor is invariant under permutations of the indices $[r]$. Note there is a slightly different but similar fact on the adjacency tensor of uniform hypergraphs. In \cite{LSQ}, Li-Shao-Qi introduced the operation of {\it moving edges} on uniform hypergraphs to increase the spectral radius.  That is, for this special symmetric nonnegative tensor with zeros on the diagonals, we have, $a_{i_1 i_2 \cdots i_r}\geq a_{j_1 j_2 \cdots j_r}$ whenever $(i_1, \ldots, i_r)\leq (j_1, \ldots, j_r)$.

However, for non-symmetric tensor, the case is different. In \cite{Shao}, Shao-Shan-Zhang proved that determinant of a tensor could change after a transpose operation on indices. 
Here we provide an example to show that the even spectral radius could be changed under transpose operation. 
 
\begin{definition}
Let $A=(a_{i_1 i_2 \cdots i_r})$ be a tensor, we call 
$M=(a_{i_1i_2\cdots i_r}')$ a transpose of $A$ if 
for all $r$-tuples $(i_1,  i_2,\ldots, i_r)$, there exists a permutation $\tau$ on $[r]$, such that
$$a_{i_1i_2\cdots i_r}'=a_{i_{\tau(1)}i_{\tau(2)}\cdots i_{\tau(r)}}.$$ 
\end{definition}
 
 When $r=2$, $\rho(M)=\rho(A)$ is always true. However, when $r\geq 3$, it is not true generally. Here is an counter-example.
Let $A$ be an 2-dimension $3$-order tensor with slices:
\[
A_1=
  \begin{bmatrix}
    1 & 2 \\
    2 & 2
  \end{bmatrix}, 
  \hspace{1cm}
  A_2=
  \begin{bmatrix}
    2 & 1 \\
    3 & 1
  \end{bmatrix}. 
\]
The spectral radius $\rho(A)=7$. 
Let $M=(a_{i_1i_2\cdots i_r}')$ be a transpose of $A$ with permutation $\tau$  such that $\tau(i)=4-i$ for $i\in [3]$.  
That is $a_{ijk}'=a_{kji}$ for any tuple $(i, j, k)$. 
Then
\[
M_1=
  \begin{bmatrix}
    1 & 2 \\
    2 & 3
  \end{bmatrix}, 
\hspace{1cm}
M_2=
  \begin{bmatrix}
    2 & 1 \\
    2 & 1
  \end{bmatrix}.
\]
However $\rho(M)=6.91618\ldots$.

\section{ A general bound on $g_r(e)$}
For a nonnegative tensor $A$, we can associate a multivariable polynomial $p_A$  as follows:
$$p_A(x_1,\ldots, x_n)=\sum_{i_1,i_2,\ldots, i_r=1}^n a_{i_1i_2\cdots i_r}x_{i_1}x_{i_2}\cdots
x_{i_r}.$$

Let $$\bar \lambda(A)=\max_{\mathbf{x}\in \mathbb{R}^n_{+}}\frac{p_A(\mathbf{x})}{\|x\|_r^r}.$$
This quality is well-defined and is closely related to $\rho(A)$. By taking $\mathbf{x}$ to
be the Perron-Frobenius vector, we have
$$\rho(A)=\frac{p_A(\mathbf{x})}{\|\mathbf{x}\|_r^r}\leq \bar \lambda(A).$$
The equality holds if $A$ is symmetric.

We call a lower dimensional tensor $B$ a principal sub-tensor of $A$ 
if $B$ consists of $m^r$ elements in $A$:   for any set $\mathbb{N}$ that composed of $m$ elements in $\{1, 2, \ldots , n\}$,
$$B=(a_{i_1\cdots i_r}), ~ \text{for all}~~ i_1, i_2, \ldots, i_r \in\mathbb{N}. $$
The concept was first introduced and used by Qi for the higher order
symmetric tensor~\cite{Qi}.

We will use several important inequalities in following sections: 
\begin{theorem}[Young's Inequality]\label{Young}
 Assume a and b are nonnegative real numbers, $p, q> 1$ and  $\frac{1}{p}+\frac{1}{q}=1$. 
 Then $$ab\leq \frac{a^p}{p}+\frac{b^q}{q}.$$
\end{theorem}
\begin{theorem}[H\"older's Inequality]\label{Holder}
Let $a_i, b_i$ be  nonnegative reals for $i=1, 2, \ldots, n$, let $p, q> 1$ and  $\frac{1}{p}+\frac{1}{q}=1$. Then 
$$\sum\limits^{n}_{i=1} a_ib_i\leq \left(\sum\limits^{n}_{i=1} a_i^{p}\right)^{1/p}\left(\sum\limits^{n}_{i=1} b_i^{q}\right)^{1/q}.$$
\end{theorem}

\begin{theorem}[Power Mean Inequality]\label{Powermean}
For nonnegative real numbers $a_1, \ldots, a_n$, if $k_1\leq k_2$, then 
$$\left(\frac{\sum\limits^{n}_{i=1} a_i^{k_1}}{n}\right)^{\frac{1}{k_1}} \leq \left(\frac{\sum\limits^{n}_{i=1} a_i^{k_2}}{n}\right)^{\frac{1}{k_2}}.$$
\end{theorem}

Let us prove Theorem \ref{t1}. 
\begin{proof}[Proof of Theorem \ref{t1}]
  Suppose that $p_A(\mathbf{x})$ reaches the maximum $\bar\lambda(A)$ at
 $\mathbf{x}=(x_i, \ldots, x_n)^{T}$ on the unit sphere under $r$-norm.
Then $\sum\limits_{i=1}^n x_i^r=1$. Using H\"older's Inequality, we have 

\begin{align} \label{youngoflamda}
\begin{split}
\bar \lambda(A)
&=\sum\limits_{i_1,i_2,\ldots, i_r=1}^{n} a_{i_1 i_2 \cdots i_r}x_{i_1} x_{i_2} \cdots x_{i_r}\\
&\leq \left(\sum\limits_{i_1,i_2,\ldots, i_r=1}^{n} (a_{i_1 i_2 \cdots i_r})^{\frac{r}{r-1}}\right)^{\frac{r-1}{r}}\left(\sum\limits_{i_1,i_2,\ldots, i_r=1}^{n} (x_{i_1}x_{i_2} \cdots x_{i_r})^r\right)^{\frac{1}{r}}\\
&= \left(\sum\limits_{i_1,i_2,\ldots, i_r=1}^{n} a_{i_1 i_2 \cdots i_r}\right)^{\frac{r-1}{r}}\times1 \\
&=e^{\frac{r-1}{r}}.
\end{split}
\end{align}
The equality holds if all $a_{i_1i_2\cdots i_r}$ are nonzeros as long as
$x_{i_1} \cdots x_{i_r}\neq 0$.
Thus $A=J_k^r$, where $J_k^r$ is a $k$-dimension $r$-order all-$1$-tensor, 
for any positive integer $k$.
\end{proof}

Here is a lower bound on $\bar\lambda(A)$.
\begin{lemma}
  If $A$ is an $n$-dimension $r$-order $\{0, 1\}$-tensor with $e$ 1's, then
  $$\bar\lambda(A)\geq \frac{e}{n}.$$
\end{lemma}
\begin{proof}
  Let $\mathbf{x}=(n^{-1/r},\ldots, n^{-1/r})$. We have
  $$\bar\lambda(A)\geq p_A(\mathbf{x})=\frac{e}{n}.$$
\end{proof}

\begin{corollary}\label{lowerbound}
  If there is a symmetric $k$-dimension $r$-order $\{0, 1\}$-tensor
  with at least $e$ 1's, then we have
  $$g_r(e)\geq \frac{e}{k}.$$

For $e=k^r-l$, $l>0$, there exists a symmetric $k$-dimension $r$-order $\{0, 1\}$-tensor
with at least $e-r!$ ones. Thus
$$g_r(k^r-l)\geq k^{r-1}-\frac{l+r!}{k}.$$

For sufficiently large  $k>r!$, we have 
$$g_r(k^r-l)\geq k^{r-1}-\frac{l}{k}.$$
This fact can be used to prove the structural theorem
for $e=k^r-l$ with small $l$.
\end{corollary}

\begin{lemma}\label{maximalisweaklyirreducible}
  If $e$ is not form of $k^r+1$ and 
   $A\in {\cal T}^r_e$ is a maximum tensor, then $A$ is weakly irreducible.
 \end{lemma}
\begin{proof}
For any integer $k$,  it is easy to verify the case when $e=k^r$.
Let $e\geq k^r+2$,   
if $A$ is not weakly irreducible,  we can re-order the elements in $[n]$ so that $A$ is a general lower-diagonal block tensor
with weakly irreducible blocks $A_1, A_2,\ldots, A_s$ on the diagonal.
Note that $\rho(A)=\rho(A_i)$ for some $i$ by Theorem \ref{decompositionoftensor}. 
If $A_i$ is not the tensor of all 1's,
we can move some $1$ to $A_i$ to get a new block $A_i'$, following by a new tensor $A'\in {\cal T}^r_e$. 
Applying Corollary \ref{subtensor},  we have $\rho(A')\geq\rho(A_i')>\rho(A_i)\geq \rho(A)$, a contradiction. 
If $A_i=J_k^r$,
and $A$ has at least two more 1's outside $A_i$, we have
$$\rho(A)=g_r(e)\geq g_r(k^r+2)>k^{r-1}=\rho(A_i).$$
Contradiction. 
\end{proof}

\textbf{Remark:} The reason we exclude the case for $e=k^r+1$ is that $g_r(k^r+1)=g_r(k^r)=k^{r-1}$, which will be proved in the last section. 
 
\begin{lemma}\label{comparewithsymmetric}
  Suppose that $A$ and $B$ are two nonnegative $n$-dimension $r$-order tensors with same number of 1's.
  Let $\mathbf{x}$ be an $H^+$-eigenvector corresponding to $\rho(A)$. If $B$ is symmetric and
  $p_A(\mathbf{x})<p_B(\mathbf{x})$, we have
  $$\rho(A)< \rho(B).$$
\end{lemma}

\begin{proof}
  Since $A\mathbf{x}^{r-1}=\rho(A) \mathbf{x}^{[r-1]}$, we have
  $$\rho(A) \|\mathbf{x}\|_r^r= p_A(\mathbf{x})<p_B(\mathbf{x})<\bar\lambda(B)\|\mathbf{x}\|_r^r.$$
  Thus $\rho(A)<\bar\lambda(B)$. Since $B$ is symmetric, we have $\bar\lambda(B)=\rho(B)$.
\end{proof}

\section{Stability results}

In this section, we will first prove a stability result; then apply it to obtain
the structure of the maximum tensors. Let us begin with the following lemma, which will be used to strengthen
the Young's inequality.

Given the same $r, e$ as in previous sections, we consider the following function:
\begin{align*}
f(x)= \frac{1}{r}x^r - \frac{1}{e^{\frac{r-1}{r}}}x + \frac{r-1}{re}.
\end{align*}

We have the following lemma. 
\begin{lemma} \label{fx}
Function $f(x)$ is continuous and $r$ times differentiable in $(0, 1)$, and has the following properties:
\begin{enumerate}
\item $f(x)\geq 0$. Equality holds if and only if $x=\frac{1}{e^{1/r}}$
\item $f(x)$ is a convex function.
\item $f(x)\geq \frac{(r-1)}{2}e^{-1+2/r}(x-e^{-1/r})^2$ for all $x>e^{-1/r}$.
\end{enumerate}

\end{lemma}
\begin{proof}
  Since $f'(x)= x^{r-1}-\frac{1}{e^{\frac{r-1}{r}}}$,  $f''(x)=(r-1)x^{r-2}\geq 0$. 
  Thus $f(x)$ is a convex function.
  By solving $f'(x)=0$ for $x$, we get the critical point $x_0=\frac{1}{e^{1/r}}$,
  thus $f(x)\geq f(x_0)=0$. For item 3, let $h(x)=f(x)-\frac{(r-1)}{2}e^{-1+2/r}(x-e^{-1/r})^2$.
  We have $h(e^{-1/r})=h'(e^{-1/r})=0$ and
   $h''(x)=(r-1)x^{r-2}-(r-1)(e^{-1/r})^{r-2}>0$ when $x>e^{-1/r}$. 
\end{proof}

Throughout this section, we will consider $r\geq 3$ as  a fixed constant, and let an integer $k$
go to infinity. 

\begin{theorem}\label{sub-tensorJ}
  Let $e=k^r+l$ where $l=o(k^{\frac{2r-2}{r^2-r+2}})$ is allowed to be either positive or negative integer.
  Let $\epsilon=0$ if $l\geq 0$ and $\epsilon=1+o_k(1)$ if $l<0$.
  For any tensor $A\in {\cal T}_e^r$ with $\rho(A)\geq k^{r-1}-\epsilon\frac{l}{k}$,
  let $v$ be the index where the Perron-Frobenius vector of $A$ reaches the maximum.
Suppose that the diagonal element $a_{v\cdots v}=1$.
Then  $A$ must contain a principal sub-tensor $A_L$ 
such that
  \begin{description}
  \item (a) There are at most $O(|l|)$ zeros in $A_L$.
  \item (b) There are at most $O(|l|)$ ones outside $A_L$.
  \item (c) The dimension of $A_L$ is $k$.
  \end{description}
\end{theorem}

The proof of this theorem is the most difficult part of the paper. We will break it into several lemmas.

Let $A$ be the tensor stated in the theorem,  $\mathbf{x}=(x_1, \ldots, x_n)^{T}$ be the Perron-Frobenius eigenvector associated to the largest eigenvector $\rho(A)$.
Assume $x_1\geq x_2\geq \cdots\geq x_n\geq 0$, and  $\sum\limits_{i=1}^n x_i^r=1$.

  Denote $I$ as the index set of all ordered $r$-tuples $(i_1, i_2,\ldots, i_r)$
  such that $a_{i_1 i_2\cdots i_r}=1$ (then $|I|=e$), and the complement $\bar{I}=[n]^r\setminus I. $ I.e. 
  $$I=\Big\{(i_1, i_2,\ldots, i_r)\in [n]^r~\vert ~a_{i_1\cdots i_r}=1 \Big\};$$
  $$\bar{I}=\Big\{(i_1, i_2,\ldots, i_r)\in [n]^r~\vert ~a_{i_1\cdots i_r}=0 \Big\}.$$

  Setting $p=\frac{r}{r-1}$, and $ q=r$, then $\frac{1}{p}+\frac{1}{q}=1$. 
  By Young's Inequality Theorem \ref{Young}, for any ordered $r$-tuple $(i_1, i_2,\ldots, i_r)\in I$,
  we have 
\begin{align*}
&\frac{a_{i_1\cdots i_r}}{\left(\sum\limits_{(j_1,\ldots, j_r)\in I}(a_{j_1\cdots j_r})^{\frac{r}{r-1}}\right)^\frac{r-1}{r}} \times  \frac{x_{i_1} \cdots x_{i_r}}{\left(\sum\limits_{(j_1,\ldots, j_r)\in I}(x_{j_1} \cdots x_{j_r})^r\right)^{\frac{1}{r}}}\\
&=\frac{1}{e^{\frac{r-1}{r}}} \times \frac{x_{i_1} \cdots x_{i_r}}{\left(\sum\limits_{(j_1,\ldots, j_r)\in I}(x_{j_1} \cdots x_{j_r})^r \right)^{\frac{1}{r}}}\\
& \leq \frac{r-1}{re} + \frac{x_{i_1}^r \cdots x_{i_r}^r}{r\sum\limits_{(j_1, \ldots, j_r)\in I}(x_{j_1} \cdots x_{j_r})^r}.
\end{align*}
   Let $$x_{i_1 i_2\cdots i_r}=\frac{x_{i_1} \cdots x_{i_r}}{\left(\sum\limits_{(j_1, \ldots, j_r)\in I}(x_{j_1} \cdots x_{j_r})^r \right)^{\frac{1}{r}}},$$ 
   then the difference of two sides in above inequality is exactly $f(x_{i_1 \cdots i_r})$:
\begin{align*}
f(x_{i_1 i_2\cdots i_r})=\frac{r-1}{re} + \frac{x_{i_1}^r \cdots x_{i_r}^r}{r\sum\limits_{(j_1, \ldots, j_r)\in I}(x_{i_1} \cdots x_{i_r})^r} - \frac{1}{e^{\frac{r-1}{r}}} \times \frac{x_{i_1} \cdots x_{i_r}}{\left(\sum\limits_{(j_1, \ldots, j_r)\in I}(x_{i_1} \cdots x_{i_r})^r \right)^{\frac{1}{r}}}.
\end{align*}

\begin{lemma}\label{l7}
  We have
  \begin{align}\label{sumfx_0+t}
\sum\limits_{(i_1, i_2, \ldots, i_r)\in I} f(x_{i_1 i_2\cdots i_r})\leq (\frac{r-1}{r}-\epsilon)\frac{l}{k^r}+ O(l^2/k^{2r}). 
\end{align}
\end{lemma}
\begin{proof}
Summing up $f(x_{i_1 i_2\cdots i_r})$ over all indexes in $I$, we get
\begin{align*}
  \sum\limits_{(i_1, \ldots, i_r)\in I} f(x_{i_1 i_2\cdots i_r})
  &=1-\frac{1}{e^{\frac{r-1}{r}}} \times \frac{\sum\limits_{(i_1, \ldots, i_r)\in I} x_{i_1} \cdots x_{i_r}}
  {\left(\sum\limits_{(i_1, \ldots, i_r)\in I}(x_{i_1} \cdots x_{i_r})^r \right)^{\frac{1}{r}}}\\
  &= 1-\frac{1}{e^{\frac{r-1}{r}}} \times \frac{\rho(A)}
  {\left(\sum\limits_{(i_1,\cdots, i_r)\in I}(x_{i_1} \cdots x_{i_r})^r \right)^{\frac{1}{r}}}.
\end{align*}

On the other hand, 
\begin{equation}\label{upperboundofsumoff(x)}
\begin{split}
  1-\sum\limits_{(i_1, \ldots, i_r)\in I} f(x_{i_1 i_2\cdots i_r})&=
\frac{1}{e^{\frac{r-1}{r}}}\frac{\rho(A)}{\left(\sum\limits_{(i_1, \ldots, i_r)\in I}(x_{i_1} \cdots x_{i_r})^{r}\right)^{\frac{1}{r}}}\\
  &\geq \frac{\rho(A)}{e^{\frac{r-1}{r}}}\\ 
   &\geq \frac{k^{r-1}+\epsilon\frac{l}{k}}{(k^{r}+l)^{\frac{r-1}{r}}}\\
&=\frac{1+ \frac{\epsilon l}{k^r}}{(1+\frac{l}{k^r})^{\frac{r-1}{r}}}\\
&= 1+(\epsilon-\frac{r-1}{r})\frac{l}{k^r} - O(l^2/k^{2r}).
\end{split}
\end{equation}
Therefore inequality \eqref{sumfx_0+t} holds.
\end{proof}

\begin{lemma}\label{rise} We have
  \begin{equation}
    \label{eq:rise}
    \sum\limits_{(i_1, \ldots, i_r)\in \bar I}x^r_{i_1} \cdots x_{i_r}^{r}
    \leq \frac{(r-\epsilon r-1)l}{k^r}+
    O(l^2/k^{2r}).
  \end{equation}
\end{lemma}

\begin{proof}
  By the Power Mean Inequality, we have
  $$\frac{\sum\limits_{(i_1, \ldots,  i_r)\in I} x_{i_1} \cdots x_{i_r}}{e}
  \leq \left(\frac{\sum\limits_{(i_1, \ldots,  i_r)\in I} x_{i_1}^r \cdots x_{i_r}^r}{e}\right)^{\frac{1}{r}}.$$

  It implies 
  \begin{align*}
  \sum\limits_{(i_1, \ldots,  i_r)\in I} x_{j_1}^r \cdots x_{j_r}^{r}
& \geq \frac{\left( \sum\limits_{(i_1, \ldots,  i_r)\in I} x_{i_1} \cdots x_{i_r}\right)^r}{e^{r-1}}\\
    &\geq \frac{\rho(A)^r}{e^{r-1}}\\
  &\geq \frac{(k^{r-1}+\epsilon l/k)^r}{(k^r+l)^{r-1}} \\
  &\geq 1 - \frac{(r-\epsilon r-1)l}{k^r}-
    O(l^2/k^{2r}).
\end{align*}
Thus,
$$\sum\limits_{(i_1, \ldots,  i_r)\in \bar I} x_{i_1}^r \cdots x_{i_r}^{r}
=1-\sum\limits_{(i_1, \ldots,  i_r)\in I} x_{i_1}^r \cdots x_{i_r}^{r}
\leq\frac{(r-\epsilon r-1)l}{k^r}+
    O(l^2/k^{2r})
.$$
\end{proof}

Now we are ready to prove Theorem \ref{sub-tensorJ}.

\begin{proof}[Proof of Theorem \ref{sub-tensorJ}]
  Let $c_2=\sqrt{(\frac{2}{r}-\frac{2\epsilon}{r-1})l}$. We claim
\begin{align}\label{x_1lessthan}
x_1^r\leq  \frac{1+c_2}{k}.
\end{align}

Otherwise, say $x_1^r>\frac{1+c_2}{k}$.  We have
 $$x_{11\cdots 1}=\frac{x_1^r}{\left(\sum\limits_{(i_1, \ldots, i_r)\in I}(x_{i_1} \cdots x_{i_r})^r \right)^{\frac{1}{r}}}\geq x_1^r > e^{-1/r} .$$ 
 
Applying Item 3 of Lemma \ref{fx},  we have 
\begin{align*}
f(x_{11\cdots 1})
&> \frac{(r-1)}{2}e^{-1+2/r}(x_{11\cdots 1}-e^{-1/r})^2 \\
& \geq \frac{(r-1)}{2}\frac{1}{(k^r+l)^{\frac{r-2}{r}}}\left(\frac{c_2}{k}\right)^2\\
&= (\frac{r-1}{r}-\epsilon)\frac{l}{k^r}- O(l^2/k^{2r}).
\end{align*}
Contradiction to inequality \eqref{sumfx_0+t} by the choice of $c_2$ as $k$ goes to infinity.

Now let $c_1$ be a constant such that : 
$$c_1= \frac{1}{2(c_2 +1)^{r-1}}.$$
We separate the index set $\{1, 2, \ldots, n\}$ into two sets $L$ and $S$, 
where $L$  is called the large set that contains element $i$ such that $x_i^r\geq \frac{c_1}{k}$,
$S$ is called the small set that contains the rest elements, i.e. 
$$L=\Big\{i\in[n] \ \vert x_i^r\geq \frac{c_1}{k}\Big\}; $$ 
$$S=\Big\{i\in[n]\  \vert x_i^r< \frac{c_1}{k}\Big\}.$$ 
Let $A_L=(a_{i_1\cdots i_r})$ be the principal sub-tensor of $A$ restricted to the large set $L$, 
i.e. for every element $a_{i_1\cdots i_r}\in A_L$,  
the index $r$-tuple $(i_1,\ldots, i_r)\in L^r$. 

Denote the number of zeros in $A_L$ as $N$. By Lemma \ref{rise}, we have
\begin{align*}
  N \times  \frac{c_1^r}{k^r} \leq
  \frac{(r-\epsilon r-1)l}{k^r}+
    O(l^2/k^{2r}).
\end{align*}

\begin{align*}
  N \leq \frac{ \frac{(r-\epsilon r-1)l}{k^r}+O(l^2/k^{2r})} {c_1^r/k^r}=  2^r(c_2+1)^{r(r-1)} \left( (r-\epsilon r-1)l
+O(l^2/k^r)
   \right).
\end{align*}

By the assumption $|l|=o(k^{\frac{2r-2}{r^2-r+2}})$, we have $N=o(k^{r-1})$.

Now consider the indexes outside of $L$. Let $(i_1\cdots i_r)\in I\setminus L^r$, by Inequality (\ref{x_1lessthan}) and the value of $c_1$, 
we have
\begin{align*}
x_{i_1\cdots i_r} &=\frac{x_{i_1} \cdots x_{i_r}}{\left(\sum\limits_{(i_1, \ldots,  i_r)\in I}(x_{i_1} \cdots x_{i_r})^r \right)^{\frac{1}{r}}}\\
&= (1+o(1))x_{i_1} \cdots x_{i_r} \\
&\leq (1+o(1)) x_1^{r-1}\left(\frac{c_1}{k}\right)^{\frac{1}{r}} \\
&\leq (1+o(1)) \frac{1}{2^{\frac{1}{r}}k}\\
&< \frac{1}{e^{1/r}}.  
\end{align*}
Note that $f(x)$ is decreasing when $x\leq  \frac{1}{e^{1/r}}$, thus 
\begin{align*}
f(x_{i_1\cdots i_r})\geq f\Big(\frac{1}{2^{\frac{1}{r}}k}\Big)\approx  \left(\frac{1}{2r} + \frac{r-1}{r}- \frac{1}{2^{\frac{1}{r}}}\right) \frac{1}{k^r} -O(l/k^{2r}). 
\end{align*} 

Let $M=|I\setminus L^r|$,  i.e. the number of 1's outside of $A_L$.  By  Inequality (\ref{sumfx_0+t}),
we have $$ M \times f\Big(\frac{1}{2^{\frac{1}{r}}k}\Big) \leq \sum_{(i_1\cdots i_r)\in I\setminus L^r}f(x_{i_1\cdots i_r})
\leq (\frac{r-1}{r}-\epsilon)\frac{l}{k^r}+ O(l^2/k^{2r}).
$$
Solving $M$, we get
$$M\leq
\frac{(\frac{r-1}{r}-\epsilon)\frac{l}{k^r}+ O(l^2/k^{2r})}
{\left(\frac{1}{2r} + \frac{r-1}{r}- \frac{1}{2^{\frac{1}{r}}}\right) \frac{1}{k^r} -O(l/k^{2r})}
=O(|l|).$$

Since the total number of 1's in tensor $A$ is $|L|^r+M-N=k^r+l$,
we have
$$|L|^r=k^r+l+N-M\leq k^r+ o(k^{r-1}).$$
Since both $|L|$ and $k$ are integers, it implies $|L|=k$.
Therefore,  the dimension of $A_L$ is $k$. 

To finish Item (a), observe
$$N=M-l= O(|l|).$$
\end{proof}

Next,  we will further determine the number of zeros in $A_L$ and number of ones outside of $A_L$ for the maximum tensors $A$ in ${\cal T}^r_{e}$. 

\begin{theorem}\label{principalsubtensorJkr}
  For fixed $r$, sufficiently large $k$, and $l>0$ a constant, let $e=k^r+l$. 
 Let  $A$ be the maximum tensor in ${\cal T}_e^r$, then $A$ contains a principal subtensor
  $J^r_k$.
\end{theorem}
\begin{proof}
Assume the dimension of $A$ is $n$. 
   Let $\rho(A)$ be the largest eigenvalue of $A$, $\mathbf{x}$ be the corresponding eigenvector, 
   with $x_1\geq x_2\geq \cdots \geq x_n.$ By Corollary \ref{lowerbound}, $\rho(A)\geq k^{r-1}$. 
   By Theorem \ref{sub-tensorJ}, $A$ contains a principal subtensor
  $A_k$ so that there are at most $O(l)$ zeros inside of $A_k$ and at most $O(l)$ ones outside of $A_k$.
 This fact implies that $x_i=(1+o(1))k^{-1/r}$ for $1\leq i \leq k$ and
  $x_j=O(k^{-1-1/r})$ for $i>k$. 
  
Here is the reason:
for any $i \in \{1, 2, \cdots, n\}$,  denote $R_i$ as the summation of elements in $i$th slice $A_i$ of $A$. 
Using H\"older’s Inequality (Theorem \ref{Holder}), we have
\begin{align*} 
\rho(A) x_i^{r-1}
&=\sum\limits_{i_2, \ldots,  i_r} a_{i i_2 \cdots i_r}x_{i_2} \cdots x_{i_r}\\ \nonumber
&\leq \left(\sum\limits_{i_2, \ldots,  i_r} (a_{i i_2 \cdots i_r})^{\frac{r}{r-1}}\right)^{\frac{r-1}{r}}\left(\sum\limits_{i_2, \ldots,  i_r} (x_{i_2} \cdots x_{i_r})^r\right)^{\frac{1}{r}}\\\nonumber
&\leq \left(\sum\limits_{i_2, \ldots,  i_r} a_{i i_2 \cdots i_r}\right)^{\frac{r-1}{r}}\times 1\\
&=R_i^{\frac{r-1}{r}}.  \nonumber
\end{align*}

For $i=1$, we have $R_1\approx k^{r-1}+o_k(1)$.  Then 
\begin{align} \label{x1upperbound}
x_1\leq \frac{R_1^{\frac{1}{r}}}{\rho(A)^{\frac{1}{r-1}}} \approx \frac{k^{\frac{r-1}{r}}+o(1)}{k} 
=\frac{1+o(1)}{k^{\frac{1}{r}}}.
\end{align}
Let $s\geq k+1$, we have $R_s\leq M\leq  O(l)$, and
 \begin{align*} 
\rho(A) x_s^{r-1}
&=\sum\limits_{i_2 \cdots i_r} a_{s i_2 \cdots i_r}x_{i_2} \cdots x_{i_r}\\ 
& \leq \sum\limits_{i_2 \cdots i_r} a_{s i_2 \cdots i_r} x_1^{r-1}\\
&= R_s x_1^{r-1}. 
\end{align*}
Then 
 \begin{align} \label{xsupperbound}
  x_s\leq\left( \frac{R_s}{\rho(A)}\right)^{\frac{1}{r-1}} x_1 \leq \left( \frac{ O(l)}{k^{r-1}}\right)^{\frac{1}{r-1}} x_1 =\frac{O(l^{\frac{1}{r-1}})}{k} x_1.
  \end{align}
Sum on $s$, we have 
$$\sum\limits_{s\geq k+1}^n x_s^r \leq  O(l) \times  \frac{O(l^{\frac{r}{r-1}})}{k^r}  x_1 ^r=o(x_1 ^r).$$
Since 
$$x_1^r+x_2^r +\cdots +x_k^r + \sum\limits_{s\geq k+1} x_s^r =1, $$
we get 
$$x_1^r\geq \frac{1-o(1)}{k}, $$ 
together with (\ref{x1upperbound}), 
$$x_1\approx \frac{1+o(1)}{k^{\frac{1}{r}}}. $$
We also have 
\begin{align*}
x_k^r
 &=1-x_1^r-\cdots-x_{k-1}^r- \sum\limits_{s\geq k+1} x_s^r\\
& \geq 1-(k-1+o(1))x_1^r\\
& \geq  1-(k-1+o(1))\frac{1}{k}\\
& \geq\frac{1-o(1)}{k}. 
\end{align*}
Then $$x_k\geq \frac{1-o(1)}{k^{\frac{1}{r}}}, $$ 
together with (\ref{x1upperbound}) and $x_k\leq x_1$, we have for any $1\leq i\leq k$, 
$$x_k\approx \frac{1+o(1)}{k^{\frac{1}{r}}}. $$
By (\ref{xsupperbound}), we have 
$$x_s\leq \frac{x_1}{k}.$$
Since $$\rho(A)x_s^{r-1}\geq x_1^{r-1},$$ 
by Theorem \ref{t1}
 $$x_s\geq \frac{x_1}{\rho(A)^{\frac{1}{r-1}}}\geq \frac{x_1}{e^{\frac{1}{r}}}\geq \frac{x_1}{k} .$$
Thus for $s\geq k+1$, $$x_s\approx \frac{x_1+o(1)}{k^{\frac{1}{r}}}\approx O(k^{-1-1/r}).$$

  We observe that the contribution to $p_A(\mathbf{x})$
  from the outside of $A_k$ is at most
  $$O(l) x_{k+1}x_1^{r-1}=\frac{O(l)}{k}x_1^r.$$
Then 
  $$p_A(\mathbf{x}) = p_{A_k}(\mathbf{x}) + p_{A- A_k}(\mathbf{x})  \leq
  p_{A_k}(\mathbf{x}) + \frac{O(l)}{k^2}. $$
If $A_k$ has some zeros, let $B=J^r_k$.
  We observe that 
  $$p_{A_k}(\mathbf{x})<p_B(\mathbf{x}).$$
  Applying Lemma \ref{comparewithsymmetric},  when $k$ is sufficiently large, we have
  $$\rho(A)<\rho(B)=k^{r-1}.$$
  Contradiction!
\end{proof}

Still let $l>0$, a similar argument can be applied to $e=k^r-l$. We have the following theorem.
\begin{theorem}\label{demensioniskfonegl}
  For fixed $r$, sufficiently large $k$, and $l>0$ a constant, let $e=k^r-l$.
Let $A$ be a maximum tensor in ${\cal T}_e^r$ with no isolated vertices.
Then the dimension of $A$ is exactly $k$.
\end{theorem}
\begin{proof}
Let $\rho(A)$ be the largest eigenvalue of $A$, $\mathbf{x}$ be the corresponding eigenvector, 
with $x_1\geq x_2\geq \cdots \geq x_n.$  By Corollary \ref{lowerbound}, $\rho(A)\geq k^{r-1}-o_k(\frac{1}{k})$. 
By a similar argument as in above theorem, we have $x_i=(1+o(1))k^{-1/r}$ for $1\leq i \leq k$ and
$x_j=O(k^{-1-1/r})$ for $i>k$. 
  
Assume there are $M>0$ ones outside of $A_k$, then there are at least $l+M$ zeros inside of $A_k$. 
Thus we have 
\begin{align*}
\rho_A(\mathbf{x})&\leq \rho_{A_k}(\mathbf{x})+ M x_{k+1}x_1^{r-1} \\
& \leq k^{r-1} -(l+M)x_k^r +  M x_{k+1}x_1^{r-1}\\
& =k^{r-1} -\frac{l+M}{k} +  \frac{M}{k^2} \\
& \leq k^{r-1} -\frac{l}{k}. 
\end{align*}
Contradicts to Corollary \ref{lowerbound} for sufficiently large $k$. Since there is no one outside $A_L$, 
the dimension of $A$ is exactly $k$. 

\end{proof}

\section{Maximum tensors in ${\cal{T}}_e^r $ with small $l$} 

In this section, we will completely determine the maximum tensors $A$ 
in ${\cal{T}}_e^r$ for $e=k^r+l$, $0\leq l\leq r$,  
and  $e=k^r-l$,  $1\leq l\leq r+1$.

Let $\mathbf{x}=(x_1, \ldots, x_n)$ be the eigenvector associated to $\rho(A)$.
Without loss of generality, we assume that $x_1\geq x_2\cdots\geq x_n$.
The tool in Theorem \ref{maxima}  allows to shift 1's
to left in the same row to increase the spectral radius of tensor $A$.
Although we couldn't shift 1's up across rows,
for example, we cannot compare
the elements
 $a_{ii_2\cdots i_r}$ and $a_{(i+1)i_2\cdots i_r}$ in a maximum tensor. But as to 
$\{0, 1\}$-tensors,  we have the following easy fact:
\begin{corollary}\label{compareiandi+1}
Let $A$ be a maximum $n$-dimension $r$-order $\{0, 1\}$-tensors. 
If $a_{j1\cdots 1}=1$ for some $j$,  then $a_{i1\cdots 1}=1$ for all $i<j$. 
\end{corollary}

\begin{proof}
Let $A_i=(a_{i i_2\cdots i_r})$ with $a_{i i_2\cdots i_r}\in A$.
If there exist $j>i$ such that $a_{i1\cdots 1}=0$ while $a_{j1\cdots 1}=1$, 
by Theorem \ref{maxima}, every other element in $A_i$ is $0$. 
Thus we have $$\rho(A)x_i^{r-1}=A_i\mathbf{x}^{r-1}=0,$$
implying $x_i=0$. However since $a_{j1\cdots 1}=1$, we have 
$$A_j\mathbf{x}^{r-1}=\rho(A)x_j^{r-1}> 0,$$
implying  $x_j> 0$, contradiction to $x_j\leq x_i$. 
\end{proof} 

By Theorem \ref{maxima} and Corollary  \ref{compareiandi+1}, we have the following property for the maximum tensor $A$ 
in ${\cal{T}}_e^r $: in each slice $A_i$, the `1' elements are always to the left and above of 
 the `0' elements. 
 
 For $e=k^r+l$, we have proved that the maximum tensor $A$ contains $J_k^r$ as principal sub-tensor (see Theorem \ref{principalsubtensorJkr}), so we just need to 
determine the positions for the rest of $l$ ones outside of $J_k^r$. 

For $l=0$, $A=J_k^r$. For $l=1$, the maximum tensor is not unique. No matter where to put the additional $1$, the resulting tensor $A$ is not
weakly irreducible. Thus it will not increase the spectral radius. We have
$$g_r(k^r+1)=g_r(k^r)=k^{r-1}.$$

For $2\leq l \leq r$, it is sufficient to prove the following facts regarding the maximum tensor $A$:
\begin{lemma}\label{threeconditionsforl>0}
For the maximum tensor $A\in {\cal{T}}_e^r$ with $e=k^r+l$, $2\leq l\leq r$,  we have
\begin{enumerate}
\item There is no `1' element in slice  $A_{k+2}$, i.e. $a_{(k+2)11\cdots 1}$ must be $0$. 
\item There is only one `1' element in slice $A_{k+1}$, which is $a_{(k+1)11\cdots 1}=1$.
\item There is no `1' elements in slice $A_i$ but outside $J_k^r$ for $i\geq 2$,  i.e. $a_{ii_2\cdots i_r}=1$ if there exists $i_j\geq k+1$.
\end{enumerate}
\end{lemma}
The details of the proof for Lemma \ref{threeconditionsforl>0} are in Appendix. 
By above analysis and  Lemma \ref{threeconditionsforl>0}, one can easily verify Item 1 and Item 2 in Theorem \ref{completlycharctersmalll}.  \\

For $e=k^r-l$, we have proved that the dimension of $A$ with $e$ ones is exactly $k$ (see Theorem \ref{demensioniskfonegl}), thus we only need to 
determine the positions for punching $l$ 0's in $A$.
For $l=1$ or $l=r+1$, we have the following results for part of Item 3 in Theorem \ref{completlycharctersmalll}.
\begin{corollary} \label{maximumtensor}
Let $r\geq 3$, $k\geq 1$ be positive integers.
\begin{enumerate}
\item Let $e=k^r-1$, the maximum tensor in ${\cal{T}}_e^r $ is obtained from $J_k^r$ by putting zero 
at $a_{kk\cdots k}$. 
\item Let $e=k^r-r-1$, the maximum tensor in ${\cal{T}}_e^r $ is obtained from $J_k^r$ by  placing zeros 
at $a_{(k-1)k\cdots k}$ and $a_{k(k-1)\cdots k},  \cdots, a_{kk\cdots (k-1)}$ and $a_{kk\cdots k}$. 
\end{enumerate}
\end{corollary}

\begin{proof}
The indicated tensor $A$ in each case is a symmetric tensor in ${\cal{T}}_e^r $, 
for any other tensor $A'\in {\cal T}^r_e$, let $\mathbf{x}$ be the vector corresponding to $\rho(A')$
with $x_1\geq x_2\geq \cdots \geq x_{k}.$ By comparing two formulas $\rho_{A'}(x)$ and $\rho_{A}(x)$, 
we can see that $\rho_{A'}(x)\leq \rho_{A}(x)$. 
Thus by Lemma \ref{comparewithsymmetric}, 
$\rho(A')\leq \rho(A)$ .
\end{proof}

For $2\leq l \leq r$, we need to prove the following lemma:
\begin{lemma}\label{0atlastslicewithotherslice}
Let  $A\in {\cal T}^r_e$  be the tensor that `0' elements appear at the end of  slices $A_{k-1}$ and $A_k$, let $B\in {\cal T}^r_e$ be the tensor that `0' elements only appear at the end of slice $A_{k}$.
Then $\rho(B)\geq \rho(A)$. 
\end{lemma}
The details of the proof of Lemma \ref{0atlastslicewithotherslice} can be found at Appendix.

Now let us we prove Theorem \ref{completlycharctersmalll}, Item 3.
\begin{proof}[Proof of Theorem \ref{completlycharctersmalll}, Item 3. ]
The idea in Lemma \ref{0atlastslicewithotherslice} is to compare the tensor when `0' elements only appear at the slice $A_k$ with tensor when `0' elements also appear at  slice $A_{k-1}$. 
Following this idea in Lemma \ref{0atlastslicewithotherslice},  we repeatedly 
compare the tensor when `0' elements only appear at slice $A_k$ with tensor when `0' elements also appear at slice $A_{k-i}$, for $i\geq 2$. There are only finite cases. It is tedious to include all computations
here.  The proof and result for each comparing is similar with the proof of 
Lemma \ref{0atlastslicewithotherslice}. In the end we conclude: For a maximum tensor $A$, the $l$ `0' elements can only appear at slice $A_k$. The proof is complete. 
\end{proof}

\section{Appendix}

\begin{proof}[Proof of Lemma \ref{threeconditionsforl>0}]

\textbf{For Item 1:}
Suppose $a_{(k+2)11\cdots 1}=1$,
by Corollary \ref{compareiandi+1}, $a_{(k+1)11\cdots 1}=1$.
By Lemma \ref{maximalisweaklyirreducible}, $A$ is a weakly irreducible tensor,  
then $a_{1(k+2)\cdots 1)}$ must be $1$.
Let $$R=\{a_{1(k+1)\cdots 1}, a_{11(k+1)\cdots 1}, \ldots, a_{11\cdots (k+1)}\}.$$  
The above assumption will force each element in set $R$ is one, then the number of ones outside of $J_k^r$ 
would be $3+r-1=r+2\geq l$, a contradiction. 
To see this, 
assume there are $s< |R|=r-1$ ones in set $R$. 
Let $\lambda$ be the largest eigenvalue of $A$, $\mathbf{x}$ be the corresponding eigenvector, 
with $x_1\geq x_2\geq \cdots \geq x_{k+2}.$ 
Then we have
\begin{align*}
&\lambda x_1^{r-1}=(x_1+\cdots +x_k)^{r-1}+ s x_{k+1}x_1^{r-2}+ x_{k+2}x_1^{r-2}\\\
&\lambda x_2^{r-1}=(x_1+\cdots +x_k)^{r-1}\\
&\cdots\\
&\lambda x_k^{r-1}=(x_1+\cdots +x_k)^{r-1}\\
&\lambda x_{k+1}^{r-1}=x_1^{r-1}\\
&\lambda x_{k+2}^{r-1}=x_1^{r-1}.
\end{align*}
Note that $x_1$ is strictly greater than $x_2$. 
Let $B$ be a new tensor obtained from $A$ by
moving `1' from $a_{(k+2)1\cdots 1}$ to one of `0' elements in set $R$. 
Let $\mathbf{y}$ be a $(k+1)-$vector obtained from $\mathbf{x}$ such that $y_i=x_i$, for $1\leq i\leq k+1$. 
By comparing $B\mathbf{y}^{r-1}$ and $\lambda \mathbf{y}^{r-1}$, we have the following system:
\begin{align*}
&\lambda y_1^{r-1}<(y_1+\cdots +y_k)^{r-1}+ (s+1) y_{k+1}y_1^{r-2}\\\
&\lambda y_2^{r-1}=(y_1+\cdots +y_k)^{r-1}\\
&\lambda y_3^{r-1}=(y_1+\cdots +y_k)^{r-1}\\
&\cdots\\
&\lambda y_{k+1}^{r-1}=y_1^{r-1}.
\end{align*}
Thus $B\mathbf{y}^{r-1}\geq \lambda \mathbf{y}^{r-1}$.
By Lemma \ref{rho(A)>k}, we have 
$\rho(B)>\lambda$, a contradiction.
Therefore $a_{(k+2)11\cdots 1}$ must be $0$,  it follows the dimension of $A$ is at most $k+1$.

\textbf{For Item 2:}
Suppose $a_{(k+1)12\cdots 1}=0$. 
Let $\lambda$ be the largest eigenvalue of $A$, $\mathbf{x}$ be the corresponding eigenvector, 
with $x_1\geq x_2\geq \cdots \geq x_{k+1}.$
Then 
\begin{align*}
&\lambda x_1^{r-1}=(x_1+\cdots +x_k)^{r-1}+ (l-2)x_{k+1}x_1^{r-2}\\
&\lambda x_2^{r-1}=(x_1+\cdots +x_k)^{r-1}\\
&\cdots\\
&\lambda x_k^{r-1}=(x_1+\cdots +x_k)^{r-1}\\
&\lambda x_{k+1}^{r-1}=x_1^{r-1}+ x_1^{r-2}x_2.
\end{align*}
Let $B$ be a tensor obtained from $A$ by moving `1' from $a_{(k+1)12\cdots 1}$ to some `0' elements in set $R$, here $R$ is defined as above. 
Let $\mathbf{y}$ be a new vector obtained from $\mathbf{x}$ such that $y_i=x_i$ for $1\leq i\leq k$,
 and $y_{k+1}=2^{-\frac{1}{r-1}}x_{k+1}$. 
Then we have 
 \begin{align*}
&\lambda y_1^{r-1}<(y_1+\cdots +y_k)^{r-1}+ (l-1)y_{k+1}y_1^{r-2}\\
&\lambda y_2^{r-1}=(y_1+\cdots +y_k)^{r-1}\\
&\lambda y_3^{r-1}=(y_1+\cdots +y_k)^{r-1}\\
&\cdots\\
&\lambda y_{k+1}^{r-1}<y_1^{r-1}.
\end{align*}
To see above system, we only need to verify the first and the last inequalities.  
Note that $x_2< x_1$, then
$$\lambda y_{k+1}^{r-1}=  \lambda (2^{-\frac{1}{r-1}}x_{k+1})^{r-1} =\frac{1}{2}x_1^{r-1}+ 
\frac{1}{2}x_1^{r-2}x_2 < x_1^{r-1}=y_1^{r-1}. $$
For the first inequality, we need to show that 
$$\lambda y_1^{r-1}=(x_1+\cdots +x_k)^{r-1}+ (l-2)x_{k+1}x_1^{r-2} 
<(y_1+\cdots +y_k)^{r-1}+ (l-1)y_{k+1}y_1^{r-2}. $$
It is equivalent to show
$$l-2 < 2^{-\frac{1}{r-1}}(l-1).$$
Let $f(r)=\frac{l-2}{l-1}-2^{-\frac{1}{r-1}}$, it is decreasing on $r$, 
then $f(r)\leq f(l)$. Since $f(l)$ is increasing on $l$, and $\lim_{l\rightarrow \infty}f(l)=0$, then 
$f(r)<0$ for all $r\geq l$. 
Now we have $B\mathbf{y}^{r-1}\geq \lambda \mathbf{y}^{r-1}$, 
by Lemma \ref{rho(A)>k}, we get $\rho(B)> \lambda$. A contradiction.

\textbf{For Item 3:}
Without loss of generality, 
we assume there are $s$ ones in $A_1$ and $t$ ones in $A_2$. 
Let $\lambda$ be the largest eigenvalue of $A$, $\mathbf{x}$ be the corresponding eigenvector, 
with $x_1\geq x_2\geq \cdots \geq x_{k+1}.$
Then 
\begin{align*}
&\lambda x_1^{r-1}=(x_1+\cdots +x_k)^{r-1}+ sx_{k+1}x_1^{r-2}\\
&\lambda x_2^{r-1}=(x_1+\cdots +x_k)^{r-1}+ tx_{k+1}x_1^{r-2}\\
&\lambda x_3^{r-1}=(x_1+\cdots +x_k)^{r-1}\\
&\cdots\\
&\lambda x_{k+1}^{r-1}=x_1^{r-1}.
\end{align*}
Replace $x_{k+1}$ by $x_1$, and let $z=(x_1+\cdots +x_k)$,  the above system is equivalent to the following:
\begin{align*}
&\lambda x_1^{r-1}=z^{r-1}+ s \lambda^{-\frac{1}{r-1}}x_1^{r-1}\\
&\lambda x_2^{r-1}=z^{r-1}+ t\lambda^{-\frac{1}{r-1}}x_1^{r-1}\\
&\lambda x_3^{r-1}=z^{r-1}\\
&\cdots\\
&\lambda x_k^{r-1}=z^{r-1}.
\end{align*}
Solve $x_1$ and $x_2$ in above system, we get 
$$x_1=(\lambda-s\lambda^{-\frac{1}{r-1}})^{-\frac{1}{r-1}}z,  \ x_2=(\frac{1-(s-t)\lambda^{-\frac{r}{r-1}}}{\lambda-s\lambda^{-\frac{1}{r-1}}})^{\frac{1}{r-1}}z. $$

Let $B$ be a tensor obtained from $A$ by moving these $t$ `0' elements in $A_2$ to set $R$. 
Still apply Lemma \ref{rho(A)>k},
we want to find a new vector  $\mathbf{y}$
such that $B\mathbf{y}^{r-1} \geq \lambda \mathbf{y}^{r-1}. $
I.e.
\begin{equation}\label{By>lambday}
\begin{split}
&\lambda y_1^{r-1}\leq z^{r-1}+ (s+t) \lambda^{-\frac{1}{r-1}}y_1^{r-1}\\ 
&\lambda y_2^{r-1}=z^{r-1}\\
&\lambda y_3^{r-1}=z^{r-1}\\
&\cdots\\
&\lambda y_k^{r-1}=z^{r-1}\\
&\lambda y_{k+1}^{r-1}=y_1^{r-1}.
\end{split}
\end{equation}

Let $y_i=x_i$ for $3\leq i\leq k+1$. 
Let $y_2= \lambda^{-\frac{1}{r-1}}z$, then
$\lambda y_2^{r-1}= z^{r-1}$.

Let $y_1=x_1+x_2-y_2$, to verify the first inequality we need to show
$y_1< (\lambda-(s+t)\lambda^{-\frac{1}{r-1}})^{-\frac{1}{r-1}}z$. 
I.e.
\begin{align*}
&\left((\lambda-s\lambda^{-\frac{1}{r-1}})^{-\frac{1}{r-1}}\right)z + \left(\frac{1-(s-t)\lambda^{-\frac{1}{r-1}}}{\lambda-s\lambda^{-\frac{1}{r-1}}}\right)^{\frac{1}{r-1}}z - \lambda^{-\frac{1}{r-1}}z \\
& <  \left(\lambda-(s+t)\lambda^{-\frac{1}{r-1}}\right)^{-\frac{1}{r-1}}z. 
\end{align*}
After divided by $\lambda^{-\frac{1}{r-1}}z$ from both sides and further simplification by letting $w=\lambda^{-\frac{r}{r-1}}>0$,  
the above inequality is equivalent to 
\begin{align*}
F(w)=\left\lbrace(1-sw)^{-\frac{1}{r-1}} -(1-(s+t)w)^{-\frac{1}{r-1}}\right\rbrace+ \left\lbrace \left(1+\frac{tw}{1-sw}\right)^{\frac{1}{r-1}}-1\right\rbrace< 0, 
\end{align*}

By Cauchy's Mean Value Theorem, we have 
\begin{align*}
F(w)=-\frac{1}{r-1}(1-\alpha)^{-\frac{1}{r-1}-1} (tw) + 
\frac{1}{r-1}(1+ \beta)^{\frac{1}{r-1}-1} \frac{tw}{1-sw} .
\end{align*}
where $sw<\alpha < (s+t)w$ and $0< \beta< \frac{tw}{1-sw}$.
Since $-(1-\alpha)^{-\frac{1}{r-1}-1}$ and $(1+ \beta)^{\frac{1}{r-1}-1}$ are decreasing functions 
on $\alpha$ and $\beta$ respectively, 
we have
\begin{align*}
F(w)& < -\frac{1}{r-1}(1-sw)^{-\frac{1}{r-1}-1} (tw) + 
\frac{1}{r-1}(1+ 0)^{\frac{1}{r-1}-1} \frac{tw}{1-sw}\\
&=  -\frac{1}{r-1} \frac{tw}{(1-sw)^{\frac{r}{r-1}}} + \frac{1}{r-1} \frac{tw}{1-sw}\\
&= \frac{tw}{(r-1)(1-sw)^{\frac{r}{r-1}}} \left((1-sw)^{\frac{1}{r-1}}-1\right)\\
& <0.
\end{align*}
Now the system (\ref{By>lambday}) is verified, by Lemma \ref{rho(A)>k}, we have 
$\rho(B)>\lambda $. A contradiction. 
\end{proof}

\begin{proof}[Proof of Lemma \ref{0atlastslicewithotherslice}.]
Let $A$ and $B$ be given tensors as stated in the lemma. 
Suppose there are $t+1$ zeros in $A_{k-1}$ and $s+1$ zeros in $A_{k}$, i.e. 
Suppose there are $t$ zeros in the set of $\{a_{(k-1)(k-1)k\cdots k}, \ldots,a_{(k-1)kk\cdots (k-1)}\}$ and one zero at $a_{(k-1)kk\cdots k}$ in slice $A_{k-1}$;  
and there are $s$ zeros in the set of 
$\{a_{k(k-1)k\cdots k}, \ldots,a_{kkk\cdots (k-1)}\}$ and one zero at $a_{kkk\cdots k}$ in slice $A_{k}$.

Let $\lambda$ be largest eigenvalue of $A$, $\mathbf{x}$ be the corresponding eigenvector with $x_1\geq x_2\geq \cdots \geq x_{k}.$
Then $s+t+2=l$, and $s\geq t\geq 0$. 
We have
\begin{align*}
&\lambda x_1^{r-1}=(x_1+\cdots +x_k)^{r-1}\\
&\ldots \\
&\lambda x_{k-2}^{r-1}=(x_1+\cdots +x_k)^{r-1}\\
&\lambda x_{k-1}^{r-1}=(x_1+\cdots +x_k)^{r-1}- tx_{k-1}x_k^{r-2}-x_k^{r-1}\\
&\lambda x_{k}^{r-1}=(x_1+\cdots +x_k)^{r-1}- sx_{k-1}x_k^{r-2}-x_k^{r-1}.
\end{align*}
From the last two equations, we get
$\lambda (x_{k-1}^{r-1}- x_{k}^{r-1})=(s-t)x_{k-1}x_k^{r-2}$. 
Let $w=\frac{x_{k-1}}{x_k}$, we have $\lambda(w^{r-1}-1)=(s-t)w$.
Let $z=(x_1+\cdots +x_k)$,
since $\lambda x_{k}^{r-1}=z^{r-1}- (sw+1)x_{k}^{r-1}$, then $ x_{k}=(\lambda+sw+1)^{-\frac{1}{r-1}}z$
and $ x_{k-1}=w(\lambda+sw+1)^{-\frac{1}{r-1}}z$. 

Note $B$ is the tensor with all zeros in the following set 
$$\{a_{k(k-1)k\cdots k}, \ldots,a_{kkk\cdots (k-1)}, a_{kkk\cdots k}\}.$$

Still apply Lemma \ref{rho(A)>k},
we want to find a new vector  $\mathbf{y}$
such that $B\mathbf{y}^{r-1} \geq \lambda \mathbf{y}^{r-1}. $
Specifically, 
\begin{equation}\label{2By>lambday}
\begin{split}
&\lambda y_1^{r-1}=(y_1+\cdots +y_k)^{r-1}\\ 
& \ldots\\
&\lambda y_{k-2}^{r-1}=(y_1+\cdots +y_k)^{r-1}\\
&\lambda y_{k-1}^{r-1}= (y_1+\cdots +y_k)^{r-1}\\
&\lambda y_k^{r-1}<(y_1+\cdots +y_k)^{r-1}-(s+t+1)y_{k-1}y_k^{r-2}-y_k^{r-1}.
\end{split}
\end{equation}

Let $y_i=x_i$ for $1\leq i\leq k-2$. 
Let $y_{k-1}= \lambda^{-\frac{1}{r-1}}z$, then
$\lambda y_{k-1}^{r-1}= z^{r-1}$. 
Let $y_k=x_k+x_{k-1}-y_{k-1}=(\lambda+sw+1)^{-\frac{1}{r-1}}z (1+w)- \lambda^{-\frac{1}{r-1}}z$.

Clearly, $\lambda y_i^{r-1}=\lambda x_i^{r-1}=x_1+\cdots +x_k=y_1+\cdots +y_k$, for $i\leq k-1$.   

We only need to verify the last inequality in system (\ref{2By>lambday}). 
I.e.
\begin{align*}
&\lambda\left((\lambda+sw+1)^{-\frac{1}{r-1}}z (1+w)- \lambda^{-\frac{1}{r-1}}z\right)^{r-1} 
\\ 
& +(s+t+1)\lambda^{-\frac{1}{r-1}}z \left((\lambda+sw+1)^{-\frac{1}{r-1}}z (1+w)- \lambda^{-\frac{1}{r-1}}z\right)^{r-2} \\
& < z^{r-1}. 
\end{align*}
After divided by $z^{r-1}$,  we have
\begin{align} \label{star}
&\lambda\left((\lambda+sw+1)^{-\frac{1}{r-1}}(1+w)- \lambda^{-\frac{1}{r-1}}\right)^{r-1} \\ \nonumber
& +(s+t+1)\lambda^{-\frac{1}{r-1}} \left((\lambda+sw+1)^{-\frac{1}{r-1}} (1+w)- \lambda^{-\frac{1}{r-1}}\right)^{r-2} \\\nonumber
& < 1\nonumber. 
\end{align}
Since $w=\frac{x_{k-1}}{x_k}\geq 1$, when $s=t$, $w=1$,   it is easy to verify that 
the left hand-side of (\ref{star}) is increasing on $\lambda$ and goes to $1$ as $\lambda\rightarrow \infty$.  Thus inequality (\ref{star}) is verified. We consider the case $s>t$,  
so $w> 1$. Let $w=1+\epsilon$, we have 
\begin{align*}
 \lambda((1+\epsilon)^{r-1}-1)=(s-t)(1+\epsilon)
\end{align*}
Solve for $\epsilon$, we get $\epsilon\approx \frac{s-t}{\lambda(r-1)}$,  
then $w\approx 1+ \frac{s-t}{\lambda(r-1)}$ . 
Then
\begin{align*}
Y& =(\lambda+sw+1)^{-\frac{1}{r-1}}(1+w)- \lambda^{-\frac{1}{r-1}}\\
&= \lambda^{-\frac{1}{r-1}}\left( \frac{1+w}{(1+\frac{sw+1}{\lambda})^{\frac{1}{r-1}}}-1\right)\\
& \approx \lambda^{-\frac{1}{r-1}} \left(w-\frac{(sw+1)(1+w)}{\lambda(r-1)}\right)\\
& \leq  \lambda^{-\frac{1}{r-1}} \left(w-\frac{2(s+1)}{\lambda(r-1)}\right)\\
& \approx \lambda^{-\frac{1}{r-1}} \left( 1-   \frac{s+t+2}{\lambda(r-1)}\right). 
\end{align*}
Then insert $Y$ to the left hand-side of (\ref{star}), we get
\begin{align*}
&\lambda Y^{r-1}+(s+t+1)\lambda^{-\frac{1}{r-1}} Y^{r-2}\\
& =  \left( 1-   \frac{s+t+2}{\lambda(r-1)}\right)^{r-1} + \frac{(s+t+1)}{\lambda}\left( 1-   \frac{s+t+2}{\lambda(r-1)}\right)^{r-2} \\
& \approx \left( 1-   \frac{s+t+2}{\lambda}\right) +  \frac{(s+t+1)}{\lambda}\left( 1-   \frac{(s+t+2)(r-2)}{\lambda(r-1)}\right)+ O(\frac{1}{\lambda^2})\\
& = 1-   \frac{1}{\lambda} + O(\frac{1}{\lambda^2})\\
& <1.
\end{align*}
Thus inequality (\ref{star}) is verified. By Lemma \ref{rho(A)>k}, we have 
$\rho(B)\geq \lambda $. A contradiction.
\end{proof}

\end{document}